\documentclass[11pt]{article}
\usepackage{amsfonts}
\usepackage{amssymb}
\usepackage{amsthm}
\usepackage{amsmath}
\usepackage{graphicx}
\usepackage{empheq}
\usepackage{indentfirst}
\usepackage{cite}
\usepackage{mathrsfs}
\usepackage{graphics}
\usepackage{enumerate}
\usepackage{color}

 \newtheorem{theorem}{Theorem}[section]
 \newtheorem{corollary}{Corollary}[section]
 \newtheorem{lemma}{Lemma}[section]
 \newtheorem{proposition}{Proposition}[section]

 \newtheorem{remark}{Remark}[section]
 \numberwithin{equation}{section}

\allowdisplaybreaks[4]

\def\e{\varepsilon}

\newcommand{\beq}{\begin{equation}}
\newcommand{\eeq}{\end{equation}}
 \def\non{\nonumber }
\def\bea{\begin{eqnarray}}
\def\eea{\end{eqnarray}}

\begin{document}
\title{Boundedness and Exponential Stabilization in a Parabolic-Elliptic Keller--Segel Model with Signal-dependent Motilities for Local Sensing Chemotaxis}

\author{Jie Jiang\thanks{Innovation Academy for Precision Measurement Science and Technology, CAS,
		Wuhan 430071, HuBei Province, P.R. China,
		\textsl{jiang@wipm.ac.cn, jiang@apm.ac.cn}.}}
\date{\today}

\maketitle

\begin{abstract} 
In this paper we consider the initial Neumann boundary value problem for a degenerate Keller--Segel model which features a signal-dependent non-increasing motility function. The main obstacle of analysis comes from the possible degeneracy when the signal concentration becomes unbounded. In the current work, we are interested in boundedness and exponential stability of the classical solution in higher dimensions. With the aid of a Lyapunov functional and a delicate Alikakos--Moser type iteration, we are able to establish a time-independent upper bound of the concentration provided that the motility function decreases algebraically. Then we further prove the uniform-in-time boundedness of the solution by constructing  of an estimation involving a weighted energy. Finally, thanks to the Lyapunov functional again, we prove the exponential stabilization toward the spatially homogeneous steady states. Our boundedness result improves those in \cite{Anh19,FJ20a} and the exponential stabilization is obtained for the first time.

\end{abstract}

\section{Introduction}	
Chemotaxis  is  a biased movement of cells due to a chemical gradient which plays a significant role in
diverse biological phenomena. In the 1970s,  Keller and Segel proposed in their seminal work \cite{KSb} the following model for chemotaxis:
 \begin{equation}\label{ksv}
 \begin{cases}
 u_t=\nabla \cdot(\gamma(v)\nabla u-u\chi(v)\nabla v),\\
 \varepsilon v_t=\Delta v-v=u.
 \end{cases}
 \end{equation}
Here, $u$ and $v$ denote the density of cells and the concentration of signals, respectively.  The signal-dependent  cell diffusion rate $\gamma$ and chemo-sensitivity $\chi$ are linked via
 \begin{equation}\label{ksv0}
 \chi(v)=(\sigma-1)\gamma'(v),
 \end{equation}
where the parameter $\sigma\geq0$ is a constant proportional to the distance between chemical receptors in the cells. In the case $\sigma>0$,  a cell determines its moving direction due to a gradient sensing mechanism by calculating the difference  of concentrations at different spots, while in  the case $\sigma=0$, the distance between receptors is zero and thus chemotactic movement occurs because of an undirected effect on activity due to the presence of a chemical sensed by a single receptor (local sensing). One notices that in the latter case, the first equation of \eqref{ksv} has the following concise form
\begin{equation}\label{ksvls}
	u_t=\Delta (\gamma(v)u)
\end{equation}
where $\gamma$ stands for a signal-dependent motility. 	More recently in \cite{PRL12,Sciencs11} , by adding a logistic source term on the right hand side of  \eqref{ksvls}, this model was also applied to describe the process of pattern formations via the so-called ``self-trapping" mechanism, where the cellular motility $\gamma(\cdot)$ was assumed to be suppressed by the concentration of signals.  In other words, $\gamma(\cdot)$ is a  signal-dependent decreasing function, i.e., $\gamma'(v)< 0$. We remark that $\gamma'(v)< 0$ indicates that cells are attracted by high concentration of signals.

In this paper, we are interested in the boundedness and stability of classical solutions to the initial boundary value problem for the parabolic-elliptic simplification of the original Keller--Segel model with signal-dependent motility for local sensing chemotaxis, i.e., $\varepsilon=\sigma=0$ in \eqref{ksv}-\eqref{ksv0}:
\begin{equation}
\begin{cases}\label{chemo1}
u_t=\Delta (\gamma (v)u)&x\in\Omega,\;t>0,\\
-\Delta v+v=u&x\in\Omega,\;t>0,\\
\partial_\nu u=\partial_\nu v=0,\qquad &x\in\partial\Omega,\;t>0,\\
u(x,0)=u_0(x),\qquad & x\in\Omega
\end{cases}
\end{equation}where $\Omega\subset\mathbb{R}^n$ with $n\geq3$  being a smooth bounded domain and
\begin{equation}\label{ini}
u_0\in C^0(\overline\Omega),\quad u_0\geq0, \quad u_0\not\equiv0.
\end{equation}
In general we require that
\begin{equation}\label{gamma0a}
\mathrm{(A0)}:\gamma(\cdot)\in C^3[0,+\infty),\;\gamma(\cdot)>0,\;\;\gamma'(\cdot)\leq0\;\;\text{on}\;(0,+\infty),\;\;\lim\limits_{s\rightarrow+\infty}\gamma(s)=0.
\end{equation}

In view of the asymptotically vanishing property of $\gamma$, an apparent difficulty in analysis lies in the possible degeneracy when $v$ becomes unbounded. Theoretical analysis for the above Keller--Segel model with signal-dependent motility has  attracted a lot of interest in recent years, see e.g., \cite{Anh19,JKW18,LW2020,WW2019,YK17,TaoWin17,BLT20,FJ19a,FJ19b,FJ20a,LiJiang,JW20}. 
A common strategy  used in most  literature is to  derive the $L^\infty_tL^p_x$-boundedness of $u$  with some $p>\frac{n}{2}$ by energy method. Then $L^\infty_tL^\infty_x$-boundedness of $v$ follows via an application of standard elliptic/parabolic regularity theory to the equation for $v$. However, this method seems only to work under restrictive conditions, for example, specific choices of $\gamma$ \cite{Anh19,JW20,YK17}, or with presence of logistic source terms \cite{JKW18,LW2020,WW2019}.
Recently, a new comparison approach  was proposed in \cite{FJ19a,FJ19b,FJ20a,LiJiang}. By introducing an auxiliary elliptic problem that enjoys a comparison principle, explicit point-wise upper bound estimates of $v$ were established directly for generic motility functions satisfying $\mathrm{(A0)}$ in any spatial dimension. In fact, it was proved that $v(x,t)$  grows at most exponentially in time and hence degeneracy cannot happen in finite time. In addition, a delicate  Alikakos--Moser type iteration was further developed in \cite{FJ20a} to deduce the uniform-in-time boundedness of $v$ directly in higher dimensions  without the help of any integrability of $u$.

Previous studies on \eqref{chemo1} strongly indicate that the dynamics of solution is closely related to the decay rate of $\gamma$. When $n=2$, it was proved in \cite{FJ19a}
that classical solution always exists globally with any large initial datum and generic $\gamma$ satisfying $\mathrm{(A0)}$. Furthermore, if $\gamma$ satisfies additionally 
\begin{equation}
\label{A2'}
\lim\limits_{s\rightarrow+\infty}e^{\alpha s}\gamma(s)=+\infty,\;\;\forall\;\alpha>0,
\end{equation}then for any large initial datum, there is a unique global classical solution which is uniformly-in-time bounded \cite{FJ20a}. Note assumption \eqref{A2'} allows $\gamma$ to take any decreasing form within a finite region and moreover, any motility function decreases slower  than a standard exponentially decreasing function at the high concentrations will guarantee the boundedness,  for example, $\gamma(s)=e^{-\sqrt{s}}$. If $\gamma$ decays even faster such that there is $\chi>0$ \begin{equation}
\label{A2''}
\lim\limits_{s\rightarrow+\infty}e^{\chi s}\gamma(s)=+\infty,
\end{equation}then the solution of \eqref{chemo1} is uniformly-in-time bounded provided that $\|u_0\|_{L^1(\Omega)}<\frac{4\pi}{\chi}$ \cite{FJ20a}. In particular, if $\gamma(s)=e^{-\chi s}$, a novel critical-mass phenomenon was observed in \cite{FJ19a} that with any sub-critical mass, the global solution is uniformly-in-time bounded while on the other hand, the global solution may blow up at time infinity with certain super-critical mass (see \cite{FJ19b,JW20,BLT20} for the fully parabolic case).

In higher dimensions $n\geq3$, boundedness was studied in several work provided that $\gamma$ satisfies some algebraically decreasing assumptions \cite{FJ20a,Wang20,Anh19}. 
In particular, if  $\gamma(s)=s^{-k}$ with some $k>0$, one notices  a variant form  of \eqref{chemo1} reads
\begin{equation}\label{variant2}
 \begin{cases}
u_t=\nabla\cdot\left[v^{-k}(\nabla u-ku\nabla \log v)\right],\\
-\Delta v+v=u,
\end{cases} 
\end{equation}
which resembles the logarithmic  Keller--Segel model:
\begin{equation}\label{logKS}
\begin{cases}
u_t=\nabla\cdot (\nabla u-ku\nabla \log v),\\
-\Delta v+v=u.
\end{cases}
\end{equation}
The above two systems share the same set of equilibria. Besides, they has the same scaling structure. Indeed, for a solution $(u,v)$, one easily checks that $(u_\lambda,v_\lambda)$ is also a solution, where $u_\lambda(t,x)=\lambda u(t,x)$ and $u_\lambda(t,x)=\lambda u(t,x)$ with any $\lambda>0$. Such a scaling invariance indicates that existence results are usually independent of the size of initial datum.

There are very limited theoretical research on both \eqref{variant2} and \eqref{logKS}. Roughly speaking, the dynamics of solution seem to be determined by the size of $k$. For the logarithmic Keller--Segel system \eqref{logKS},  there are several study on  the admissible range of $k$ for global existence/boundedness and on the other hand, blowup solution was constructed only  in the radial symmetric case   $n\geq3$ and $k>\frac{2n}{n-2}$ \cite{Nagai98}. The threshold number is still unclear. We refer the reader to  \cite{FS2016,fs2018, LanWin} for  a complete description of related topics for \eqref{logKS}. For the degenerate system \eqref{variant2}, boundedness of global solutions  with any $0<k<\frac{2}{n-2}$ was first shown in \cite{Anh19} and later in \cite{FJ20a,Wang20} via different methods. Moreover, existence of global (but likely growing up) classical solution  was obtained in \cite{FJ20a} within a larger range $0<k<\frac{\sqrt{2n}+2}{n-2}$.

In the present work, we aim to improve the uniform-in-time boundedness result for \eqref{chemo1} in higher dimensions. The key observation of this contribution is that under an assumption
\begin{equation}\label{gamma}
\mathrm{(A1)}: \gamma(s)+s\gamma'(s)\geq0,\qquad\forall\;s>0,
\end{equation}
system \eqref{chemo1} possesses a Lyapunov functional (see also \cite{Anh19}). Then we can perform an Alikakos--Moser iteration to derive a time-independent upper bound of $v$ under  weakened conditions compared with \cite{FJ20a}. Besides, with the aid of the Lyapunov functional, exponential stabilization of the global solution toward the spatially homogeneous steady states $(\overline{u_0},\overline{u_0})$ is obtained for the first time. A direct consequence of our result to  the specific case $\gamma(s)=s^{-k}$ is that  boundedness of solutions can be improved to any $k\leq1$ when $n=4,5$, or $k<\frac{4}{n-2}$ when $n\geq6$. Besides, the solution will converge to $(\overline{u_0},\overline{u_0})$ exponentially as time tends to infinity.

In order to formulate our result in a more general framework, we introduce the following condition:
 \begin{equation}\label{gamma2}\mathrm{(A2)}:\qquad\text{there is $k>0$ such that}
 \lim\limits_{s\rightarrow+\infty}s^{k}\gamma(s)=+\infty.
 \end{equation}
 Note that $\mathrm{(A2)}$ allows $\gamma$ to take other algebraically decreasing functions, for example, $\gamma(s)=\frac{1}{s^{k}\log(1+s)}$ with any $k>0$.

 Now we are in a position to state the main results of the current work.
\begin{theorem}
	\label{TH1} Assume $n\geq4$. Suppose that $\gamma$ satisfies $\mathrm{(A0)}$,  $\mathrm{(A1)}$ and
	\begin{equation}
\label{A3}\mathrm{(A3)}:\qquad l_0|\gamma'(s)|^2\leq \gamma(s)\gamma''(s),\;\text{with some }\;l_0> \frac{n+2}{4}\;
\;\text{for all}\;s>0.
\end{equation}
 Then for any initial datum satisfying \eqref{ini}, problem \eqref{chemo1} possesses  a unique  global classical solution that is uniformly-in-time bounded.
 
 Moreover, there exist $\alpha>0$ and $C>0$ depending on $u_0,\gamma,n$ and $\Omega$ such that for all $t\geq1$,
 \begin{equation}
 	\|u(\cdot,t)-\overline{u_0}\|_{L^\infty(\Omega)}+\|v(\cdot,t)-\overline{u_0}\|_{W^{1,\infty}(\Omega)}\leq Ce^{-\alpha t}
 \end{equation}where $\overline{u_0}=\frac{1}{|\Omega|}\int_\Omega u_0dx.$
\end{theorem}
\begin{remark}
Thanks to the strictly positive time-independent lower bound $v_*$ of $v$ for $(x,t)\in\overline{\Omega}\times[0,\infty)$ given in  Lemma \ref{lbdv} in the next section, assumptions  $\mathrm{(A1)}$ and  $\mathrm{(A3)}$ can be weaken as to hold for all $s\geq v_*$. Thus, our existence and boundedness results also hold true if $\gamma(s)$ has singularities at $s=0$, for example $\gamma(s)=s^{-k}$ with $k>0$. In such cases, we can simply replace $\gamma(s)$ by a new motility function $\tilde{\gamma}(s)$ which satisfies $\mathrm{(A0)}$ and coincides with $\gamma(s)$ for $s\geq\frac{v_*}{2}$.
\end{remark}

In particular, for the typical case $\gamma(v)=v^{-k}$, we have
\begin{theorem}
	Suppose  that $\gamma(v)=v^{-k}$ and $n\geq4$. Then for any initial datum satisfying \eqref{ini},
problem \eqref{chemo1} has a unique global classical solution which is uniformly-in-time bounded provided that $0<k\leq 1$ when $n=4,5$, or $0<k<\frac{4}{n-2}$ when $n\geq6.$ Moreover, there exist $\alpha>0$ and $C>0$ depending on $u_0,k,n$ and $\Omega$ such that for all $t\geq1$,
\begin{equation}\label{exst}
\|u(\cdot,t)-\overline{u_0}\|_{L^\infty(\Omega)}+\|v(\cdot,t)-\overline{u_0}\|_{W^{1,\infty}(\Omega)}\leq Ce^{-\alpha t}.
\end{equation}	
\end{theorem}
\begin{remark}
Combined with the results in \cite{Anh19,FJ20a}, uniform-in-time boundedness for the case $\gamma(s)=s^{-k}$ is now available if
\begin{equation}
0<	k\;\;\begin{cases}
	<\infty,\qquad &n=2,\\
	<2,\qquad &n=3,\\
	\leq 1,\qquad &n=4,5,\\
	<\frac{4}{n-2},\qquad &n\geq6.
	\end{cases}
\end{equation}
And the exponential decay \eqref{exst} also holds when $n\leq 3$ if $0<k\leq1$. We remark that the convergence of $(u,v)$ toward the constant solution was found in \cite{Anh19} when $\gamma(s)=s^{-k}$ supposing that $k\in(0,\frac{2}{n-2})\cap (0,1]$. However, no convergence rate was given.
\end{remark}
Now, let us sketch the main  idea of our proof. First, we would like to recall the following  identity which unveils the key mechanism of our system:
\begin{equation*}
	v_t+u\gamma(v)=(I-\Delta)^{-1}[u\gamma(v)].
\end{equation*}Here $\Delta$ denotes the usual Neumann Laplacian operator. The above key identity was first observed in \cite{FJ19a,FJ19b} which along with a new comparison approach, gives rise to a point-wise upper bound of $v$ with generic functions satisfying  $\mathrm{(A0)}$. Furthermore, one notices that a substitution of the second equation of \eqref{chemo1} gives a variant form of this key identity:
\begin{equation}\label{keyidvar}
	v_t-\gamma(v)\Delta v+v\gamma(v)=(I-\Delta)^{-1}[u\gamma(v)].
\end{equation} Thanks to the comparison principle of elliptic equations and the decreasing property of $\gamma$, one has
\begin{equation*}
	(I-\Delta)^{-1}[u\gamma(v)]\leq \gamma(v_*)(I-\Delta)^{-1}[u]= \gamma(v_*)v
\end{equation*}with $v_*$ being the strictly positive lower bound for $v$ given by Lemma \ref{lbdv} below.
Then under the assumption $\mathrm{(A2)}$, based on a delicate Alikakos--Moser type iteration, we can show that uniform-in-time upper bound of $v$ is obtainable if we have time-independent  estimates for $\sup_{t\geq0}\|v\|_{L^q}$ with any $q>\frac{nk}{2}$ beforehand. 

On the other hand, system \eqref{chemo1} possesses  a Lyapunov functional (see also \cite{Anh19}) such that
\begin{equation}\label{Lyaintro1}
	\frac{1}{2}\frac{d}{dt}\left(\|\nabla v\|^2+\|v\|^2\right)+\int_\Omega \gamma(v)|\Delta v|^2dx+\int_\Omega (\gamma(v)+v\gamma'(v))|\nabla v|^2dx=0
\end{equation}
which implies a  time-independent estimate of $\sup_{t\geq0}\|v\|_{H^1(\Omega)}$ under the assumption $\mathrm{(A1)}$. Then  the Sobolev embedding $H^1\hookrightarrow L^{q_*}$  with $q_*=\frac{2n}{n-2}$ yields to a time-independent estimates for $\sup_{t\geq0}\|v\|_{L^{q_*}}$, which together with Alikakos--Moser iteration indicates that $v$ is uniform-in-time bounded provided that
$q_*>\frac{nk}{2}$, i.e., $k<\frac{4}{n-2}$. 

Next, in order to prove the boundedness of solutions, it suffices to establish $L^\infty_tL^p_x$-boundedness of $u$ with some $p>\frac{n}{2}$ since higher-order estimates can be then proved by standard iterations and bootstrap argument. Recalling that $v$ is now bounded from above, $\gamma(v)$ is bounded from below by a strictly positive time-independent constant due to its decreasing property. With the aid of the key identity  again, we construct an estimation involving a weighted energy $\int_\Omega u^p\gamma^q(v)$, which with proper choice of $p>\frac{n}{2}$ and $q>0$ will finally imply the boundedness.

Last, the Lyapunov functional also plays a crucial role in the study of exponential stabilization. Since $\overline{v(t)}=\overline{u_0}$, we have $\int_\Omega v_tdx=0$ for all $t>0$ and hence   the energy-dissipation relation  \eqref{Lyaintro1} can be rewritten as 
\begin{equation}
\frac{1}{2}\frac{d}{dt}\left(\|\nabla v\|^2+\|v-\overline{u_0}\|^2\right)+\int_\Omega \gamma(v)|\Delta v|^2dx+\int_\Omega (\gamma(v)+v\gamma'(v))|\nabla v|^2dx=0.
\end{equation}With the boundedness of $v$ at hand, one can deduce from above by Poincar\'e's inequality that $\|v-\overline{u_0}\|_{H^1}$ decay exponentially. Then by a  bootstrapping strategy, exponential stabilization of $(u,v)$ can be further acquired in $L^\infty\times W^{1,\infty}$.

We remark that if the second equation of \eqref{chemo1} is of parabolic type. It is still unknown whether the system possesses a Lyapunov functional like \eqref{Lyaintro1}. Thus, at the present stage, we cannot improve the results in \cite{FJ20a} for the fully parabolic case using the same idea.

The rest of the paper is organized as follows.  In Section 2, we provide some preliminary results and recall some useful lemmas. In Section 3 we first construct a Lyapunov functional which satisfies certain dissipation property. Then  using a delicate Alikakos--Moser iteration, we  derive the uniform-in-time upper bounds of $v$. In Section 4, we first establish the boundedness of the weighted energy which gives rise to the boundedness of the global classical solutions. Then using the Lyapunov functional again we prove the exponential stabilization toward the constant steady states.

\section{Preliminaries}	
	In this section, we recall some useful lemmas. First, local existence and uniqueness of classical solutions to system \eqref{chemo1} can be
	established by the standard fixed point argument and  regularity theory for elliptic/parabolic equations. Similar proof can be found in \cite[Lemma 3.1]{Anh19} and hence here we omit the detail here.
	\begin{theorem}\label{local}
		Let $\Omega$ be a smooth bounded domain of $\mathbb{R}^n$. Suppose that $\gamma(\cdot)$ satisfies $\mathrm{(A0)}$ and $u_0$ satisfies \eqref{ini}. Then there exists $T_{\mathrm{max}} \in (0, \infty]$ such that problem \eqref{chemo1} permits a unique non-negative classical solution $(u,v)\in (C^0(\overline{\Omega}\times[0,T_{\mathrm{max}}))\cap C^{2,1}(\overline{\Omega}\times(0,T_{\mathrm{max}})))^2$. Moreover, the following mass conservation holds
		\begin{equation*}
		\int_{\Omega}u(\cdot,t)dx=\int_{\Omega}v(\cdot,t)dx=\int_{\Omega}u_0 dx
		\quad \text{for\ all}\ t \in (0,T_{\mathrm{max}}).
		\end{equation*}	
		If $T_{\mathrm{max}}<\infty$, then

		\begin{equation*}
	\limsup\limits_{t\nearrow T_{\mathrm{max}}} 
		\|u(\cdot,t)\|_{L^\infty(\Omega)}=\infty.
		\end{equation*}
		\end{theorem}


A strictly positive uniform-in-time lower bound for $v=(I-\Delta)^{-1}[u](x,t)$ is given in \cite[Lemma 2.2]{Anh19}; see also \cite[Lemma 3.3]{Black}.
	\begin{lemma}\label{lbdv}
		Suppose $(u,v)$  is the classical solution of \eqref{chemo1} up to the maximal time of existence $T_{\mathrm{max}}\in(0,\infty]$. Then, there exists a strictly positive constant $v_*=v_*(n,\Omega,\|u_0\|_{L^1(\Omega)})$ such that for all $t\in(0,T_{\mathrm{max}})$, there holds
		\begin{equation*}
		\inf\limits_{x\in\Omega}v(x,t)\geq v_*.
		\end{equation*}
	\end{lemma}
Next, we recall the following key identity and an explicit point-wise  upper bound estimate for  $v$ \cite[Lemma 3.1]{FJ19a}.
\begin{lemma}\label{keylem1}Assume $n\geq1$ and suppose that $\gamma$ satisfies  $(\mathrm{A0})$.  For any $0<t<T_{\mathrm{max}}$, there holds
	\begin{equation}\label{keyid}
	v_t+\gamma(v)u=(I-\Delta)^{-1}[\gamma(v)u].
	\end{equation}
	Moreover, for  any $x\in\Omega$ and  $t\in[0,T_{\mathrm{max}})$, we have
	\begin{equation}\label{ptesta}
	v(x,t)\leq v_0(x)e^{\gamma(v_*)t}
	\end{equation}with $v_0\triangleq(I-\Delta)^{-1}[u_0].$
\end{lemma}	
Last, we recall the following $L^p-L^q$ estimates for the Neumann heat semigroup on bounded domains (see e.g., \cite{Cao,Win10}).
\begin{lemma}\label{lmpq}
	Suppose $\{e^{t\Delta}\}_{t\geq0}$ is the Neumann heat semigroup in $\Omega$, and $\mu_1>0$ denote the first nonzero eigenvalue of $-\Delta$ in $\Omega$ under Neumann boundary conditions. Then there exist $k_1, k_2>0$ which only depend on $\Omega$ such that the following properties hold:
	\begin{enumerate}[(i)]
		\item If $1\leq q\leq p\leq \infty,$ then
		\begin{equation}
		\|e^{t\Delta}w\|_{L^p(\Omega)}\leq k_1(1+t^{-\frac{d}{2}(\frac1q-\frac1p)})e^{-\mu_1t}\|w\|_{L^q(\Omega)}\qquad\text{for all}\;\;t>0
		\end{equation}for all $w\in L^q_0(\Omega)\triangleq\{w\in L^q(\Omega),\;\;\int_\Omega wdx=0\}$;
		\item  If $1<q\leq p\leq \infty,$ then
		\begin{equation}
		\|e^{t\Delta}\nabla \cdot w\|_{L^p(\Omega)}\leq k_2(1+t^{-\frac12-\frac{d}{2}(\frac1q-\frac1p)})e^{-\mu_1 t}\|w\|_{L^q(\Omega)}\quad\text{for all}\;\;t>0
		\end{equation}for any $w\in (W^{1,p}(\Omega))^d.$
	\end{enumerate}
\end{lemma}

\section{Time-independent upper bounds of $v$}

In this part, we aim to establish uniform-in-time upper bound for $v$ in higher dimensions when $\gamma$ decreases algebraically at large concentrations. The proof of the above result consists of several steps. To begin with, we introduce a Lyapunov functional.
\begin{lemma}\label{wH1b}
For any $0\leq t<T_{\mathrm{max}}$, there holds
\begin{equation}\label{Lya0}
\frac{1}{2}\frac{d}{dt}\left(\|\nabla v\|^2+\|v\|^2\right)+\int_\Omega \gamma(v)|\Delta v|^2dx+\int_\Omega (\gamma(v)+v\gamma'(v))|\nabla v|^2dx=0.
\end{equation} In particular, under the assumption $\mathrm{(A1)}$, there is $C>0$ depending only on $u_0$ such that
\begin{equation}\label{veh1}
		\sup\limits_{0\leq t<T_{\mathrm{max}}}\bigg(\|\nabla v\|^2+\|v\|^2\bigg)\leq C.
\end{equation}
\end{lemma}
	\begin{proof}
Multiplying the first equation of \eqref{chemo1} by $v$, integrating over $\Omega$ and substituting the second equation of \eqref{chemo1} yields that
\begin{equation*}
	\frac{1}{2}\frac{d}{dt}\left(\|\nabla v\|^2+\|v\|^2\right)=\int_\Omega u\gamma(v)\Delta vdx=\int_\Omega \gamma(v)\Delta v(v-\Delta v)dx.
\end{equation*}
By integration by parts,  there holds
\begin{equation*}
	\begin{split}
	\int_\Omega \gamma(v)v\Delta vdx=-\int_\Omega \nabla v\cdot\nabla (v\gamma(v))dx
	=-\int_\Omega (\gamma(v)+v\gamma'(v))|\nabla v|^2.
	\end{split}
\end{equation*}
Thus, we obtain that
\begin{equation*}
\frac{1}{2}\frac{d}{dt}\left(\|\nabla v\|^2+\|v\|^2\right)+\int_\Omega \gamma(v)|\Delta v|^2dx+\int_\Omega (\gamma(v)+v\gamma'(v))|\nabla v|^2dx=0.
\end{equation*}
Then uniform-in-time estimate \eqref{veh1} follows by integration of above identity with respect to time. This completes the proof.
	\end{proof}
\begin{remark}
	In view of the time-independent lower bound $0<v_*\leq v(x,t)$, one can slightly weaken assumption $\mathrm{(A1)}$ as 
	\begin{equation}
		\gamma(s)+s\gamma'(s)\geq0,\,\qquad\forall\;s\geq v_*.
	\end{equation}
On the other hand, a direct calculation indicates that the above assumption yields
\begin{equation}
	s\gamma(s)\geq v_*\gamma(v_*),\,\qquad\forall\;s\geq v_*
\end{equation}and hence  $\gamma$ fulfills $\mathrm{(A2)}$ with any $k>1$. In particular, if $\gamma(s)=s^{-k}$, assumption $\mathrm{(A1)}$ is satisfied with any $k\leq1.$

\end{remark}

With the above result, we can establish the uniform-in-time upper bounds of $v$ based on a delicate  Alikakos--Moser iteration \cite{Alik79}.  First, we show that
\begin{lemma}\label{lmvbd0pre1}
	Assume $n\geq3$. Suppose $\gamma$ satisfies $\mathrm{(A0)}$ and $\mathrm{(A2)}$ with some $k>0$.  
	Then there exist  $\lambda_1,\lambda_2>0$ independent of time
	such that for any  $p> 1+k$,
	\begin{equation}
	\frac{d}{dt}\int_\Omega v^{p}
+\lambda_2 p\int_\Omega v^p	
	+\frac{\lambda_1 p(p-k-1)}{(p-k)^2}\int_\Omega|\nabla v^{\frac{p-k}{2}}|^2+\lambda_1 p\int_\Omega v^{p-k}
	\leq  2\lambda_2 p\int_\Omega v^p.
	\end{equation}
\end{lemma}

\begin{proof} First,  under the our assumption,  we may infer that there exist  $b>0$ and $s_b>v_*$ such that for all $s\geq s_b$
	\begin{equation*}
	1/\gamma(s)\leq bs^k
	\end{equation*}and on the other hand, since $\gamma(\cdot)$ is non-increasing,
	\begin{equation*}
	1/\gamma(s)\leq 1/\gamma(s_b)
	\end{equation*}for all $0\leq s<s_b$.
	Therefore, for all $s\geq0$, there holds
	\begin{equation}\label{cond_gamma}
	1/\gamma(s)\leq bs^{k}+1/\gamma(s_b).
	\end{equation}
	Now, multiplying the key identity \eqref{keyid} by $v^{p-1}$ with some $p>1+k$, we obtain that
	\begin{equation}
	\frac{1}{p}\frac{d}{dt}\int_\Omega v^p+\int_\Omega u\gamma(v)v^{p-1}=\int_\Omega (I-\Delta)^{-1}[u\gamma(v)]v^{p-1}.
	\end{equation}	Since $\gamma(v)\leq \gamma(v_*)$, we deduce by the comparison principle of elliptic equation that 
	\begin{equation*}
		(I-\Delta)^{-1}[u\gamma(v)]\leq \gamma(v_*)v
	\end{equation*}and hence
	\begin{equation*}
		\int_\Omega (I-\Delta)^{-1}[u\gamma(v)]v^{p-1}\leq\gamma(v_*)\int_\Omega v^p.
	\end{equation*}
	Thanks to \eqref{cond_gamma}, it follows   that
	\begin{equation*}
	\begin{split}
	\int_\Omega u\gamma(v)v^{p-1}dx\geq&\int_\Omega u \bigg(bv^{k}+1/\gamma(s_b)\bigg)^{-1}v^{p-1}dx\\
	\geq&C\int_\Omega (v^{k}+1)^{-1}v^{p-1}udx
	\end{split}
	\end{equation*}with $C>0$ independent of $p$ and time.
	Since $v^k\geq v_*^k$ by Lemma \ref{lbdv}, there holds
	\begin{equation}
		\begin{split}
		(v^k+1)^{-1}v^{p-1}\geq(v^k+v_*^{-k}v^k)^{-1}v^{p-1}=\frac{v^{p-k-1}}{1+v_*^{-k}}
		\end{split}
	\end{equation}
from which we deduce that	
		\begin{equation}\label{gammab}
	\begin{split}
	\int_\Omega u\gamma(v)v^{p-1}dx\geq C\int_\Omega v^{p-k-1}udx
	\end{split}
	\end{equation}where $C>0$ may depend on the initial datum, $n,\Omega$ and $\gamma$, but is  independent of $p$ and time.
	Next, recalling that $v-\Delta v=u$, we observe that
	\begin{equation*}
	\begin{split}
	\int_\Omega v^{p-k-1}udx=&\int_\Omega v^{p-k-1}(v-\Delta v)dx\\
	=&\int_\Omega v^{p-k}dx+(p-k-1)\int_\Omega|\nabla v|^2v^{p-k-2}\\
	=&\int_\Omega v^{p-k}dx+\frac{4(p-k-1)}{(p-k)^2}\int_\Omega|\nabla v^{\frac{p-k}{2}}|^2.
	\end{split}
	\end{equation*}
	Therefore, we arrive at
	\begin{equation*}
	\frac{d}{dt}\int_\Omega v^{p}+\frac{\lambda_1 p(p-k-1)}{(p-k)^2}\int_\Omega|\nabla v^{\frac{p-k}{2}}|^2+\lambda_1 p\int_\Omega v^{p-k}\leq  \lambda_2 p\int_\Omega v^p
	\end{equation*}with some $\lambda_1,\lambda_2>0$ independent of $p$ and time. 
This completes the proof by adding $\lambda_2 p\int_\Omega v^p$ to both sides of the above inequality.
\end{proof}

\begin{lemma}\label{lmvbd0pre2}
	Assume $n\geq3$. Suppose $\gamma$ satisfies $\mathrm{(A0)}$ and $\mathrm{(A2)}$ with some $0<k<\frac{4}{n-2}$.  
Let $L>1$ be a generic constant.
There exists $C_0>0$ depending only on the initial datum, $\Omega,k$ and $n$ such that for any $p >q\geq q_*=\frac{2n}{n-2}$ satisfying 
$$
q<p=2q-\frac{nk}{2},
$$ 
there holds
\begin{equation}
\frac{d}{dt}\int_\Omega v^{p}+\lambda_2p\int_\Omega v^p\leq  C_0L^{\frac{n}{2}}p^{\frac{n+2}{2}}\left(\int_\Omega v^q\right)^2.
\end{equation}
\end{lemma}
	
\begin{proof} First, one notices that $q_*>1+k$   and  $q_*>\frac{nk}{2}$ provided that $n\geq3$ and $0<k<\frac{4}{n-2}$. Let 
$$
q_*\leq q<p=2q-\frac{nk}{2} = 2q-\frac{kq_*}{q_*-2}.
$$
Denote $\eta=v^{\frac{p-k}{2}}$ and  define
\begin{equation}\label{alpha}
\alpha=\frac{(p-k)(p-q)}{p(p-k-2q/q_*)}.
\end{equation}	
One easily checks that $\alpha\in(0,1)$. Indeed, 
\begin{equation*}
	\begin{split}
	p-k-\frac{2q}{q_*}>&q-\frac{2q}{q_*}-k
=\frac{q_*-2}{q_*}q-k\\
>&\frac{q_*-2}{q_*}\frac{nk}{2}-k=	\frac{q_*-2}{q_*}\frac{kq_*}{q_*-2}-k
=0
	\end{split}
\end{equation*}
and on the other hand, solving $\alpha<1$ yields $p>\frac{kq_*}{q_*-2}=\frac{nk}{2}$. Moreover, since $q>\frac{nk}{2}$ as well,  one checks that  $\frac{2p\alpha}{p-k}<2$.  Then an application of H\"older's inequality yields that
	\begin{equation*}
	\begin{split}
		\int_\Omega v^{p}dx&=\int_\Omega \eta^{\frac{2p}{p-k}}dx=\int_\Omega\eta^{\frac{2p\alpha}{p-k}}\eta^{\frac{2p(1-\alpha)}{p-k}}dx\\
		&\leq\left(\int_\Omega \eta^{q_*}dx\right)^{\frac{2p\alpha}{(p-k)q_*}}\left(\int_\Omega \eta^{\frac{2p(1-\alpha)q_*}{(p-k)q_*-2p\alpha}} \right)^{\frac{(p-k)q_*-2p\alpha}{(p-k)q_*}}\qquad(\text{since}\;\;\frac{2p\alpha}{p-k}<2<q_*)\\
		&=\| \eta\|^{\frac{2p\alpha}{p-k}}_{L^{q_*}(\Omega)}\left(\int_\Omega \eta^{\frac{2q}{p-k}}\right)^{\frac{(p-k)q_*-2p\alpha}{(p-k)q_*}}\\
			&=\| \eta\|^{\frac{2p\alpha}{p-k}}_{L^{q_*}(\Omega)}\left(\int_\Omega v^q\right)^{\frac{(p-k)q_*-2p\alpha}{(p-k)q_*}}.
	\end{split}
	\end{equation*}
Recall the Sobolev embedding inequality 
\begin{equation*}
\|\eta\|_{L^{q_*}(\Omega)}\leq \lambda_*\|\eta\|_{H^1(\Omega)}
\end{equation*}where  $\lambda_*>0$ depends only on $n$ and $\Omega$. Since $\frac{2p\alpha}{p-k}<2$, invoking Young's inequality, we infer that
\begin{equation*}
	\begin{split}
		&\lambda_2p\int_\Omega v^{p}dx\\
	\leq&\lambda_2p\| \eta\|^{\frac{2p\alpha}{p-k}}_{L^{q_*}(\Omega)}\left(\int_\Omega v^{q}\right)^{\frac{(p-k)q_*-2p\alpha}{(p-k)q_*}}\\
	\leq&\lambda_2p\bigg(\lambda_*\| \eta\|_{H^1(\Omega)}\bigg)^{\frac{2p\alpha}{p-k}}\left(\int_\Omega v^{q}\right)^{\frac{(p-k)q_*-2p\alpha}{(p-k)q_*}}\\
	\leq&\frac{p\alpha\delta^{\frac{p-k}{p\alpha}}}{p-k}\|\eta\|^2_{H^1(\Omega)}+\frac{p-k-p\alpha}{p-k}\lambda_*^{\frac{2p\alpha}{p-k-p\alpha}}\left(\delta^{-1}\lambda_2p\right)^{\frac{p-k}{p-k-p\alpha}}\left(\int_\Omega v^{q}\right)^{\frac{(p-k)q_*-2p\alpha}{(p-k-p\alpha)q_*}},
	\end{split}
\end{equation*}	
where $\delta>0$ such that
	\begin{equation}\label{choise_delta}
		\frac{p\alpha\delta^{\frac{p-k}{p\alpha}}}{p-k}=\frac{\lambda_1p(p-k-1)}{2L(p-k)^2}.
	\end{equation} 
	It follows from above and \eqref{alpha} that
\begin{equation}
	\begin{split}
	&\frac{p-k-p\alpha}{p-k}\lambda_*^{\frac{2p\alpha}{p-k-p\alpha}}\left(\delta^{-1}\lambda_2p\right)^{\frac{p-k}{p-k-p\alpha}}\left(\int_\Omega v^{q}\right)^{\frac{(p-k)q_*-2p\alpha}{(p-k-p\alpha)q_*}}\\
	=&\frac{p-k-p\alpha}{p-k}\left(\frac{2L\alpha(p-k)\lambda_*^2}{\lambda_1(p-k-1)}\right)^{\frac{p\alpha}{p-k-p\alpha}}\left(\lambda_2 p\right)^{\frac{p-k}{p-k-p\alpha}}\left(\int_\Omega v^{q}\right)^{\frac{(p-k)q_*-2p\alpha}{(p-k-p\alpha)q_*}}\\
	=&\frac{(q_*-2)q-kq_*}{q_*(p-k)-2q}\left(\frac{2L(p-k)^2(p-q)\lambda_*^2}{\lambda_1 p(p-k-1)(p-k-2q/q_*)}\right)^{\frac{(p-q)q_*}{q(q_*-2)-kq_*}}(\lambda_2p)^{\frac{q_*(p-k)-2q}{(q_*-2)q-kq_*}}\\
	&\times\left(\int_\Omega v^q\right)^{\frac{q_*(p-k)-2p}{q_*(q-k)-2q}}.
	\end{split}
\end{equation}	
Noticing that
\begin{equation*}
\frac{n}{2}=\frac{q_*}{q_*-2},
\end{equation*} and recalling that 
$$
\frac{nk}{2}<q<p=2q-\frac{nk}{2} = 2q-\frac{kq_*}{q_*-2},
$$ 
one easily checks that
 \begin{eqnarray*}
 \frac{q_*(p-k)-2p}{q_*(q-k)-2q}&=&2,\\
\frac{(q_*-2)q-kq_*}{q_*(p-k)-2q}&=&\frac{2}{n+2},\\
\frac{(p-q)q_*}{q(q_*-2)-kq_*}&=&\frac{n}{2},\\
\frac{q_*(p-k)-2q}{(q_*-2)q-kq_*}&=&\frac{n+2}{2},
\end{eqnarray*}
 and
 \begin{equation*}
 	\frac{p-q}{p-k-2q/q_*}=\frac{n}{n+2}.
 \end{equation*}
Moreover, since $p>q_*>1+k$,
\begin{equation}\non
	\begin{split}
\frac{(p-k)^2}{p(p-k-1)}=&\frac{(p-k-1)^2+2(p-k-1)+1}{p(p-k-1)}\\
=&\frac{p-k-1}{p}+\frac{2}{p}+\frac{1}{p(p-k-1)}\\
<&3+\frac{1}{p-k-1}\\
<&3+\frac{1}{q_*-k-1},
	\end{split}
\end{equation}and since $p>\frac{nk}{2}$
\begin{equation}\label{llll}
	\frac{(p-k)^2}{p(p-k-1)}>\frac{p-k}{p}=1-\frac{k}{p}>1-\frac{2}{n}=\frac{n-2}{n}>0.
\end{equation}
Based on  the above calculations, it follows that
\begin{eqnarray*}
\frac{2L(p-k)^2(p-q)\lambda_*^2}{\lambda_1 p(p-k-1)(p-k-2q/q_*)}
&=&\frac{2L\lambda_*^2}{\lambda_1}
\cdot \frac{n}{n+2}
\cdot
\frac{(p-k)^2}{p(p-k-1)}\\
&<&\frac{2Ln\lambda_*^2}{\lambda_1(n+2)}
\left(3+\frac{1}{q_*-k-1}\right).
\end{eqnarray*}
Hence one can find   $C_0>0$ being a constant depending only on the initial datum, $\Omega,k$ and $n$ such that
\begin{eqnarray*}
&&\frac{(q_*-2)q-kq_*}{q_*(p-k)-2q}
\left(\frac{2L(p-k)^2(p-q)\lambda_*^2}{\lambda_1 p(p-k-1)(p-k-2q/q_*)}\right)^{\frac{(p-q)q_*}{q(q_*-2)-kq_*}}
(\lambda_2p)^{\frac{q_*(p-k)-2q}{(q_*-2)q-kq_*}}\\
&<&\frac{2}{n+2} \cdot
\left\{\frac{2Ln\lambda_*^2}{\lambda_1(n+2)}
\left(3+\frac{1}{q_*-k-1}\right)
 \right\}^{\frac{n}{2}}
(\lambda_2p)^{\frac{n+2}{2}} \\
	&\leq& \frac{C_0}{2}L^{\frac{n}{2}}p^{\frac{n+2}{2}}.
\end{eqnarray*}
Therefore by the above and \eqref{choise_delta} we have
\begin{eqnarray*}
		2\lambda_2p\int_\Omega v^{p}dx
	\leq
\frac{\lambda_1p(p-k-1)}{L(p-k)^2}
	\|v^{\frac{p-k}{2}}\|^2_{H^1(\Omega)}
	+ C_0L^{\frac{n}{2}}p^{\frac{n+2}{2}}\left(\int_\Omega v^{q}\right)^{2}.
\end{eqnarray*}	
Combining Lemma \ref{lmvbd0pre1} and recalling $L>1$,
 we obtain the following inequality
\begin{equation}\non
\frac{d}{dt}\int_\Omega v^{p}+\lambda_2p\int_\Omega v^p\leq  C_0L^{\frac{n}{2}}p^{\frac{n+2}{2}}\left(\int_\Omega v^q\right)^2.
\end{equation}
\end{proof}

\begin{proposition}\label{Propunbv}
	Assume $n\geq4$. Suppose $\gamma$ satisfies $\mathrm{(A0)}$, $\mathrm{(A1)}$ and  $\mathrm{(A2)}$ with some $k\in(0,\frac{4}{n-2})$. Then there is $v^*>0$ depending only on the initial datum, $\gamma,n$ and $\Omega$ such that
	\begin{equation}
	\sup\limits_{0\leq t<T_{\mathrm{max}}} \|v(\cdot,t)\|_{L^\infty(\Omega)}\leq v^*.
	\end{equation}
\end{proposition}
\begin{proof}
For all $r\in\mathbb{N}$ we define 
$$p_r\triangleq2^{r}(q_*-\frac{nk}{2})+\frac{nk}{2},\qquad p_0=q_*.$$ 
Then $p_r>q_*> \frac{nk}{2}$ and $p_r=2p_{r-1}-\frac{nk}{2}$. 
We apply Lemma \ref{lmvbd0pre2} with $(p,q)=(p_r, p_{r-1})$ to have 
\begin{equation}\non
\frac{d}{dt}\int_\Omega v^{p_r}+ \lambda_2 p_r\int_\Omega v^{p_r}\leq  \lambda_2p_r\mathcal{A}_r
\left(\mathcal{M}_{r-1}  \right)^{2},
\end{equation}
where
\begin{equation*}
\mathcal{M}_{r}\triangleq\sup\limits_{0\leq t< T_{\mathrm{max}}}\int_\Omega v^{p_r}
\quad
\mbox{and}
\quad
	\mathcal{A}_r\triangleq\frac{C_0L^{\frac{n}{2}}p_r^\frac{n}{2}}{\lambda_2}.
\end{equation*}
By solving the above ODE, it follows that for all $r\in\mathbb{N}$
\begin{equation*}
\mathcal{M}_{r} = \sup\limits_{0\leq t< T_{\mathrm{max}}}\int_\Omega v^{p_r}\leq\max\{\mathcal{A}_r\mathcal{M}_{r-1}^2,\| v_0\|_{L^\infty(\Omega)}^{p_r}   \}.
\end{equation*}
Since $p_r\geq q_*$ for all $r\geq1$, one can choose $L>1$ sufficiently large depending only on the initial datum, $\Omega$, $n$ and $k$ such that $\mathcal{A}_r>1$ for all $r\geq1.$ Moreover,  adjusting $C_0$ by a proper larger number, we have
\begin{equation}
	\begin{split}\non
	\mathcal{A}_r	\leq C_0a^{r}
	\end{split}
\end{equation}with some $a>0$ depending only on the initial datum, $\Omega,$ $k$ and $n$. 
In addition, since $\gamma$ satisfies $\mathrm{(A0)}$ and $\mathrm{(A1)}$, due to Lemma \ref{wH1b} and the Sobolev embedding $H^1\hookrightarrow L^{q_*},$ we may find some large constant $K_0>1$  that dominates $\|v_0\|_{L^\infty}$ and $\int_\Omega v^{q_*}$ for all time. 

Iteratively, we deduce that
\begin{equation}
	\begin{split}
	\int_\Omega v^{p_r}\leq&\max\{\mathcal{A}_r\mathcal{A}_{r-1}^2\mathcal{M}^4_{r-2},\mathcal{A}_rK_0^{2p_{r-1}},K_0^{p_r}\}\\
	=&\max\{\mathcal{A}_r\mathcal{A}_{r-1}^2\mathcal{M}^4_{r-2},\mathcal{A}_rK_0^{2p_{r-1}}\}\\
	\leq&\dots\\
	\leq&\max\{\mathcal{A}_r\mathcal{A}_{r-1}^2\mathcal{A}_{r-2}^4\cdots\mathcal{A}_1^{2^{r-1}}\mathcal{M}_0^{2^r},\mathcal{A}_r\mathcal{A}_{r-1}^2\cdots\mathcal{A}_2^{2^{r-2}}K_0^{2^{r-1}p_1}\} \\
	\leq&\max\{\mathcal{A}_r\mathcal{A}_{r-1}^2\mathcal{A}_{r-2}^4\cdots\mathcal{A}_1^{2^{r-1}}K_0^{2^r},\mathcal{A}_r\mathcal{A}_{r-1}^2\cdots\mathcal{A}_2^{2^{r-2}}K_0^{2^{r-1}p_1}\}\\
	\leq& C_0^{2^0+2^1+\cdots+2^{r-1}}
	\times a^{1\cdot r+2(r-1)+2^2(r-2)+\cdots+2^{r-1}(r-(r-1))}\times\tilde{K}_0^{2^r}\\
	=&{ C_0^{2^r-1}} a^{2^{1+r}-r-2}\tilde{K}^{2^r}_0\non
	\end{split}
\end{equation}where $\tilde{K}=\max\{K_0,K_0^{\frac{p_1}{2}}\}$.
Finally, recalling that $p_r= 2^{r}(q_*-\frac{nk}{2})+\frac{nk}{2}$,  we deduce that
\begin{equation}\non
	\|v\|_{L^\infty(\Omega)}\leq\lim\limits_{r\nearrow+\infty}\left({ C_0^{2^r-1} } a^{2^{1+r}-r-2}\tilde{K}^{2^r}_0\right)^{1/p_r}=\left(C_0a^2\tilde{K}_0\right)^{\frac{2}{2q_*-nk}},
\end{equation}
which concludes the proof.
\end{proof}
\begin{corollary}
	If $\gamma(v)=v^{-k}$, then $v$ has a uniform-in-time upper bound provided that $k\leq 1$ when $n=4,5$, or $k<\frac{4}{n-2}$ when $n\geq6$.
\end{corollary}
Next, we recall the following lemma established in \cite{FJ20a}.
\begin{lemma}\label{lemA23}
	A function satisfying $\mathrm{(A0)}$, $\mathrm{(A1)}$ and  
	\begin{equation}\label{A3c}
	\mathrm{(A3')}:\qquad l|\gamma'(s)|^2\leq \gamma(s)\gamma''(s),\;\;\forall\;s>0
	\end{equation}
	with some $l>1$ must fulfill assumption  $\mathrm{(A2)}$ with any $k>\frac{1}{l-1}$.
\end{lemma}
\begin{proof}
	First, we point out that under the assumptions $\mathrm{(A0)}$ and $\mathrm{(A3')}$, $\gamma'(s)<0$ on $[0,\infty).$ In fact, due to $\mathrm{(A0)}$ and $\mathrm{(A3')}$, we have $\gamma''(s)\geq0$ for all $s>0$. Then if there is $s_1\geq0$ such that $\gamma'(s_1)=0$, it must hold that $0=\gamma'(s_1)\leq \gamma'(s)\leq0$ for all $s\geq s_1$, which contradicts to the positivity of $\gamma$  and the asymptotically vanishing assumption $\mathrm{(A0)}$.
	
	Now, we may divide \eqref{A3c} by $-\gamma(s)\gamma'(s)$ to obtain that
	\begin{equation*}
	-\frac{l\gamma'(s)}{\gamma(s)}\leq -\frac{\gamma''(s)}{\gamma'(s)},\;\;\;\;\forall s>0,
	\end{equation*}
	which indicates that
	\begin{equation*}
	\left(\log(-\gamma^{-l}\gamma')\right)'\leq0.
	\end{equation*}
	An integration of above ODI from $v_*$ to $s$ yields that
	\begin{equation}\non
	-\gamma^{-l}(s)\gamma'(s)\leq-\gamma^{-l}(v_*)\gamma'(v_*)\triangleq d>0,
	\end{equation}which further implies that
	\begin{equation*}
	\left(\frac{1}{(l-1)\gamma^{l-1}(s)}\right)'\leq d.
	\end{equation*}
	Thus for any $s\geq v_*$, there holds
	\begin{equation*}
	\frac{1}{\gamma^{l-1}(s)}\leq d(l-1)(s-v_*)+\frac{1}{\gamma^{l-1}(v_*)}.
	\end{equation*}
	As a result, for any $k>\frac{1}{l-1}$, we have
	\begin{equation}\non
	\frac{1}{[s^{k}\gamma(s)]^{l-1}}\leq \frac{d(l-1)(s-v_*)}{s^{k(l-1)}}+\frac{1}{s^{k(l-1)}\gamma^{l-1}(v_*)}\rightarrow0,\;\;\;\text{as}\;s\rightarrow+\infty.
	\end{equation}This completes the proof.
\end{proof}
\begin{lemma}\label{cor1a}
	Assume that $n\geq4$. Suppose that $\gamma(\cdot)$ satisfies $\mathrm{(A0)}$, $\mathrm{(A1)}$ and  $\mathrm{(A3)}$. Then $v$ has a uniform-in-time upper bound in $\overline{\Omega}\times[0,T_{\mathrm{max}}).$
\end{lemma}
\begin{proof}
	Note that $\frac{1}{l_0-1}<\frac{4}{n-2}$ since $l_0>\frac{n+2}{4}$. Thus $\gamma$ satisfies $\mathrm{(A2)}$ with some $k<\frac{4}{n-2}$ and due to Proposition \ref{Propunbv}, $v$ has a uniform-in-time upper bound such that $v\leq v^*$ on $\overline{\Omega}\times[0,T_{\mathrm{max}})$.
\end{proof}

\section{Proof of Theorem \ref{TH1}}
\subsection{Uniform-in-time boundedness}
This section is devoted to the proof of Theorem \ref{TH1}. With the time-independent upper bound of $v$ at hand, it suffices to establish an estimation involving a weighted energy $\int_\Omega u^{1+p}\gamma^q(v)$ for some $1+p>\frac{n}{2}$ and $q>0$. Higher-order estimates can be then proved via a standard bootstrapping argument.   To begin with, we show that
\begin{lemma} For any $0\leq t< T_{\mathrm{max}}$ and $p,q>0$, there holds
\begin{equation}
\begin{split}\label{est0}
&\frac{d}{dt}\int_{\Omega}u^{p+1}\gamma^q(v)dx+(p+1)p\int_{\Omega}u^{p-1}\gamma^{q+1}|\nabla u|^2dx\\
&+q\int_{\Omega}\bigg((p+q+1)|\gamma'(v)|^2+\gamma\gamma''\bigg)u^{p+1}\gamma^{q-1}|\nabla v|^2dx\\
&-q\int_{\Omega}(I-\Delta)^{-1}[u\gamma(v)]  u^{p+1}\gamma^{q-1}(v)\gamma'(v)dx\\
=&-(p+1)(p+2q)\int_{\Omega}u^p\gamma^q(v)\gamma'(v)\nabla u\cdot\nabla vdx  -q\int_{\Omega}u^{p+1}\gamma^q(v)\gamma'(v)vdx.
\end{split}
\end{equation}
\end{lemma} 
\begin{proof}
 Multiplying  the key identity \eqref{keyid} by $qu^{p+1}\gamma^{q-1}(v)\gamma'(v)$ with $p,q>0$  and integrating over $\Omega$ yields
\begin{equation}\label{e3}
\begin{split}
\frac{d}{dt}\int_{\Omega}u^{p+1}\gamma^q(v)dx-(p+1)\int_{\Omega}\gamma^q(v)u^pu_tdx-q\int_{\Omega}u^{p+1}\gamma^q(v)\gamma'(v)\Delta vdx\\
-q\int_{\Omega}(I-\Delta)^{-1}[u\gamma(v)]  u^{p+1}\gamma^{q-1}(v)\gamma'(v)dx= -q\int_{\Omega}u^{p+1}\gamma^q(v)\gamma'(v)vdx,
\end{split}
\end{equation}
where we substitute the second equation $-\Delta v+v=u$.

By integration by parts and the first equation of \eqref{chemo1}, we infer that
\begin{equation}
\begin{split}\label{inte0}
&-(p+1)\int_{\Omega}\gamma^q(v)u^pu_tdx\\
=&-(p+1)\int_{\Omega}\gamma^q(v)u^p \Delta (\gamma (v)u)dx\\
=&(p+1)\int_{\Omega}\left(\gamma (v)\nabla u+\gamma'(v)u\nabla v\right)\left(pu^{p-1}\gamma^q(v)\nabla u+qu^p\gamma^{q-1}(v)\gamma'(v)\nabla v\right)dx\\
=&p(p+1)\int_{\Omega}u^{p-1}\gamma^{q+1}(v)|\nabla u|^2dx  +q(p+1)\int_{\Omega}u^{p+1}\gamma ^{q-1}(v)|\gamma'(v)|^2|\nabla v|^2dx\\
&+(p+1)(p+q)\int_{\Omega}u^p\gamma^q(v)\gamma'(v)\nabla u\cdot\nabla vdx.
\end{split}
\end{equation}
Lat,  by integration by parts again, there holds
\begin{equation}\label{e4}
\begin{split}
&-q\int_{\Omega}u^{p+1}\gamma^q(v)\gamma'(v)\Delta vdx\\
=&q^2\int_{\Omega} u^{p+1}\gamma^{q-1} (v)|\gamma'(v)|^2|\nabla v|^2dx+q\int_{\Omega}u^{p+1}\gamma^q\gamma''(v)|\nabla v|^2dx\\
&+q(p+1)\int_{\Omega}u^p\gamma^q (v)\gamma'(v) \nabla u\cdot\nabla vdx.
\end{split}
\end{equation}
This completes the proof by collecting above equalities.
\end{proof}
 \begin{lemma}\label{lmest}
 	Assume that $\gamma(\cdot)$ satisfies $\mathrm{(A0)}$, $\mathrm{(A1)}$ and 
 	$\mathrm{(A3)}$. For any $1+p\in(0,l_0^2)$, there exist
 	time-independent constants $q=\frac{pl_0}{2}>0$ and $\delta_0=\delta_0(p,q)\in(0,1)$ such that
 	\begin{equation}\label{assert0}
 	\frac{(p+1)(p+2q)^2}{4p(1-\delta_0)}\int_\Omega u^{1+p}\gamma^{q-1}|\gamma'|^2|\nabla v|^2\leq q\int_{\Omega}\bigg((p+q+1)|\gamma'(v)|^2+\gamma\gamma''\bigg)u^{p+1}\gamma^{q-1}|\nabla v|^2.
 	\end{equation}
 \end{lemma}
 \begin{proof}
Define \[f(\lambda)=4\lambda l_0-4\lambda^2\] for all $\lambda>0$. We observe that $f(\lambda)$ attains its maximum value $l_0^2$ at $\lambda_0=l_0/2$. 
 Thus, for any $1+p\in(1,l_0^2)$, there holds
 \begin{equation*}
 1+p<l_0^2=f(\lambda_0).
 \end{equation*}In other words,
 \begin{equation*}
 \frac{1+p+4\lambda_0^2}{4\lambda_0}<l_0.
 \end{equation*}
 In addition,  we can further find time-independent $\delta_0=\delta_0(p,\lambda_0)\in(0,1)$ such that
 \begin{equation*}
 \frac{1+p+4\lambda_0^2+4\lambda_0\delta_0(1+p+\lambda_0p)}{4\lambda_0(1-\delta_0)}=\frac12\left(\frac{1+p+4\lambda_0^2}{4\lambda_0}+l_0\right)\in(\frac{1+p+4\lambda_0^2}{4\lambda_0},l_0).
 \end{equation*}
 Thus, we obtain that
 \begin{equation*}
 \frac{1+p+4\lambda_0^2+4\lambda_0\delta_0(1+p+\lambda_0p)}{4\lambda_0(1-\delta_0)}|\gamma'|^2< l_0|\gamma'|^2\leq \gamma\gamma'',\;\;\forall\;s>0
 \end{equation*}
 for any	$1+p\in(1,l_0^2)$. On the other hand, according to Lemma \ref{lbdv} and Lemma \ref{cor1a}, there exist the time-independent lower and upper bounds for $v$
 \begin{equation*}
 v_*\leq v(x,t)\leq v^* \qquad\text{on}\;\;\overline{\Omega}\times [0,T_{\mathrm{max}}).
 \end{equation*} Thus
 we infer that
 \begin{equation*}
 \frac{1+p+4\lambda_0^2+4\lambda_0\delta_0(1+p+\lambda_0p)}{4\lambda_0(1-\delta_0)}|\gamma'(v(x,t))|^2<  \gamma(v(x,t))\gamma''(v(x,t)),\;\;\text{on}\;\overline{\Omega}\times[0,T_{\mathrm{max}})
 \end{equation*}for any	$1+p\in(1,l_0^2)$. Thus,  assertion \eqref{assert0} holds with $q=\lambda_0p$ and $\delta_0$ chosen above.	
 \end{proof}
\begin{lemma}\label{lmuest}
	Assume that $n\geq3$ and $\gamma(\cdot)$ satisfies $\mathrm{(A0)}$, $\mathrm{(A1)}$ and 
	 $\mathrm{(A3)}$. Then  
	 there {exist $p>\frac{n}{2}-1$ and $C>0$} independent of time such that
	\begin{equation*}
		\sup\limits_{0\leq t< T_{\mathrm{max}}}\int_\Omega u^{1+p}\leq C.
	\end{equation*}
\end{lemma}
\begin{proof}
Invoking Young's inequality,  there holds
\begin{equation}
\begin{split}
&-(p+1)(p+2q)\int_{\Omega}u^p\gamma^q(v)\gamma'(v)\nabla u\cdot\nabla vdx\\
&\leq (p+1)p(1-\delta_0)\int_{\Omega}u^{p-1}\gamma^{q+1}|\nabla u|^2dx\\&+\frac{(p+1)(p+2q)^2}{4p(1-\delta_0)}\int_\Omega u^{1+p}\gamma^{q-1}|\gamma'|^2|\nabla v|^2dx
\end{split}
\end{equation}with any $1+p\in(1,l_0^2)$ and $q,\delta_0$ chosen in Lemma \ref{lmest}.

Then it follows from \eqref{est0} and Lemma \ref{lmest} that
\begin{equation}
\begin{split}\label{est2b}
&\frac{d}{dt}\int_{\Omega}u^{p+1}\gamma^q(v)dx+\delta_0(p+1)p\int_{\Omega}u^{p-1}\gamma^{q+1}|\nabla u|^2dx\\
&\;\;-q\int_{\Omega}(I-\Delta)^{-1}[u\gamma(v)]  u^{p+1}\gamma^{q-1}(v)\gamma'(v)dx\\
\leq& -q\int_{\Omega}u^{p+1}\gamma^q(v)\gamma'(v)vdx
\end{split}
\end{equation}with any $1+p\in(1,l_0^2)$ and $q=\frac{pl_0}{2}$.

Since $v_*\leq v\leq v^*$, there is $C>0$ depending on $p$, $\delta_0,$ $\gamma$ and the initial datum only such that for any $1+p\in(1,l_0^2)$ and $q=\frac{pl_0}{2}$
\begin{equation}
\begin{split}\label{est2bu}
\frac{d}{dt}\int_{\Omega}u^{p+1}\gamma^q(v)dx+C\int_{\Omega}u^{p-1}|\nabla u|^2dx
\leq C\int_{\Omega}u^{p+1}dx.
\end{split}
\end{equation}

Recall the Gagliardo-Nirenberg inequality
\begin{equation}\label{emb0}
\|\xi\|_{L^{2}(\Omega)}\leq C\|\nabla \xi\|_{L^2(\Omega)}^{\theta}\|\xi\|_{L^1(\Omega)}^{1-\theta}+C\|\xi\|_{L^1(\Omega)}
\end{equation}
with $\theta=\frac{n}{n+2}$. Denote $\xi=u^{\frac{1+p}{2}}$. Then in view of the uniform-in-time boundedness of $\gamma^q(v)$, we infer by Young's inequality that
\begin{equation}\non
\int_\Omega u^{1+p}\gamma^q(v)\leq C\int_\Omega u^{1+p}= C\|\xi\|_{L^2(\Omega)}^2\leq \e\|\nabla \xi\|^2+C_\e\|\xi\|^2_{L^1(\Omega)}
\end{equation}with any $\e>0$ and some $C>0$ independent of time. Thus, by choosing proper small  $\e>0$, we infer from \eqref{est2bu} that for any $1+p\in(1,l_0^2)$ and $q=\frac{pl_0}{2}$,
\begin{equation}
\begin{split}\label{est2bxi}
\frac{d}{dt}\int_{\Omega}u^{1+p}\gamma^q(v)dx+C\int_\Omega u^{1+p}\gamma^q(v)
\leq C\int_\Omega u^{\frac{1+p}{2}}.
\end{split}
\end{equation}with $C>0$ independent of time. Then in view of the fact $\|u(\cdot,t)\|_{L^1(\Omega)}=\|u_0\|_{L^1(\Omega)}$ together with the uniform-in-time lower and upper boundedness of $\gamma^q(v)$, one can deduce iteratively from \eqref{est2bxi} that for any $1+p\in(1,l_0^2)$ and $q=\frac{pl_0}{2}$, 
\begin{equation}\label{upest}
\sup\limits_{t\geq0}\int_\Omega u^{1+p}dx\leq C.
\end{equation}
 Finally, we note that for any $n\geq3$, 
$$\frac{n}{2}<\left(\frac{n+2}{4}\right)^2<l_0^2.$$ Thus we can always find $p>0$ satisfying  $1+p>\frac{n}{2}$ such that \eqref{upest} holds. This completes the proof.
\end{proof}
\noindent\textbf{Proof of Theorem \ref{TH1}. Boundedness:} Once we obtain lemma \ref{lmuest} with $1+p>\frac{n}{2}$, we can  to deduce the uniform-in-time boundedness of the solutions in the same manner as done in \cite[Lemma 4.3]{Anh19}. We omit the detail here.\qed

\begin{corollary}\label{cor1} Assume   $\gamma(v)=v^{-k}$ and $n\geq4$. Then there exists a unique globally bounded classical solution provided that $k\leq1$ when $n=4,5$, or  $k<\frac{4}{n-2}$ when $n\geq6.$
\end{corollary}
\subsection{Exponential stabilization toward constant steady states}
 In this part, we establish the exponential stabilization of the global solutions relying on a slightly modified version of the Lyapunov functional \eqref{Lya0}.
 
\begin{lemma}\label{lmv1decay}
	There exist constants $\alpha>0$ and $C>0$ depending on $u_0,\gamma,n$ and $\Omega$ such that
	\begin{equation}
		\|v(\cdot,t)-\overline{u_0}\|_{W^{1,\infty}(\Omega)}\leq Ce^{-\alpha t},\qquad\forall\;t>0.
	\end{equation}
\end{lemma} 
\begin{proof}
Observing that $\overline{u}(t)=\overline{v}(t)=\overline{u_0}$ for all $t\geq0,$ we infer that
\begin{equation*}
\begin{split}
\frac{d}{dt}\|v(\cdot,t)-\overline{u_0}\|^2=&\frac{d}{dt}\int_\Omega \left(v^2-2\overline{u_0}v+\overline{u_0}^2\right)\\
=&\frac{d}{dt}\int_\Omega v^2-2\overline{u_0}\frac{d}{dt}\int_\Omega vdx\\
=&\frac{d}{dt}\|v(\cdot,t)\|^2.
\end{split}	
\end{equation*} 
Therefore, we deduce from \eqref{Lya0} that
\begin{equation}\label{Lya1}
\frac{1}{2}\frac{d}{dt}\left(\|\nabla v\|^2+\|v-\overline{u_0}\|^2\right)+\int_\Omega \gamma(v)|\Delta v|^2dx+\int_\Omega (\gamma(v)+v\gamma'(v))|\nabla v|^2dx=0.
\end{equation}
Since $v\leq v^*$ and $\gamma$ is non-increasing , we have
\begin{eqnarray*}
\int_\Omega \gamma(v)|\Delta v|^2\geq\gamma(v^*)\int_\Omega |\Delta v|^2
\end{eqnarray*} which together with Poincar\'e's inequality $\mu_1\|v-\overline{u_0}\|^2\leq \|\nabla v\|^2$ yields that
\begin{eqnarray}
\int_\Omega \gamma(v)|\Delta v|^2\geq\mu_1\gamma(v^*)\int_\Omega |\nabla v|^2\geq \frac{\mu_1\gamma(v^*)}{1+\mu_1}\bigg(\|\nabla v\|^2+\|v-\overline{u_0}\|^2\bigg).
\end{eqnarray} Here $\mu_1>0$ denotes the first positive eigenvalue of the Neumann Lapaplacian operator and we also use the fact $\mu_1\|\nabla v\|^2\leq \|\Delta v\|^2$ if $\partial_\nu v=0$ on $\partial\Omega$. 

Thus, we infer from \eqref{Lya1} that
\begin{equation*}
\frac{1}{2}\frac{d}{dt}\left(\|\nabla v\|^2+\|v-\overline{u_0}\|^2\right)+\frac{\mu_1\gamma(v^*)}{1+\mu_1}\bigg(\|\nabla v\|^2+\|v-\overline{u_0}\|^2\bigg)\leq0,
\end{equation*}
which by standard ODI analysis yields that for all $t\geq0$
\begin{equation}\label{vdecay1}
\|\nabla v\|^2+\|v-\overline{u_0}\|^2\leq e^{-\frac{2\mu_1\gamma(v^*)t}{1+\mu_1}}\left(\|\nabla v_0\|^2+\|v_0-\overline{u_0}\|^2\right)
\end{equation}where $v_0=(I-\Delta)^{-1}[u_0]$.

Next, from the first equation of \eqref{chemo1},  we have
\begin{equation}\label{ueqb}
\partial_t(u-\overline{u_0})=\Delta (u\gamma(v)).
\end{equation}
Multiplying \eqref{ueqb} by $u-\overline{u_0}$, we obtain that
\begin{eqnarray*}
	\frac12\frac{d}{dt}\int_\Omega |u-\overline{u_0}|^2dx+\int_\Omega \gamma(v)|\nabla u|^2
	&=&\int_\Omega u\gamma'(v)\nabla u\cdot\nabla vdx\\
	&\leq&\frac{1}{2}\int_\Omega \gamma(v)|\nabla u|^2+\frac{1}{2}\int_\Omega \frac{u^2|\gamma'(v)|^2}{\gamma(v)}|\nabla v|^2.
\end{eqnarray*}
Since now $u,v$ are both uniformly-in-time bounded, there is time-independent constant $C>0$ such that
\begin{equation*}
\frac{d}{dt}\int_\Omega |u-\overline{u_0}|^2dx+C\int_\Omega |\nabla u|^2\leq C\int_\Omega |\nabla v|^2.
\end{equation*}
Applying Poincar\'e's inequality, one can find constant $0<\alpha_1<\frac{2\mu_1\gamma(v^*)}{1+\mu_1}$ depending only on initial datum, $\gamma,n$ and $\Omega$ such that
\begin{equation*}
\frac{d}{dt}\int_\Omega |u-\overline{u_0}|^2dx+\alpha_1\int_\Omega |u-\overline{u_0}|^2\leq C\int_\Omega |\nabla v|^2.
\end{equation*}In view of \eqref{vdecay1}, solving the above differential inequality yields that
\begin{equation}
\|u-\overline{u_0}\|^2\leq Ce^{-\alpha_1t}
\end{equation}with $C$ depending on initial datum, $n,\gamma$ and $\Omega$ only.

Next, for any $p>2$, we multiply \eqref{ueqb} by $|u-\overline{u_0}|^{p-2}(u-\overline{u_0})$ to get that
\begin{eqnarray*}
	&&\frac1p\frac{d}{dt}\int_\Omega |u-\overline{u_0}|^pdx+(p-1)\int_\Omega \gamma(v)|u-\overline{u_0}|^{p-2}|\nabla u|^2\\
	&=&(p-1)\int_\Omega u|u-\overline{u_0}|^{p-2}\gamma'(v)\nabla u\cdot\nabla v\\
	&\leq&\frac{p-1}{2}\int_\Omega \gamma(v)|u-\overline{u_0}|^{p-2}|\nabla u|^2+\frac{p-1}{2}\int_\Omega \frac{u^2|\gamma'(v)|^2}{\gamma(v)}|u-\overline{u_0}|^{p-2}|\nabla v|^2.
\end{eqnarray*}
Similarly, there is time-independent constants $C=C(p)>0$ and $\alpha_2<\alpha_1$ such that
\begin{equation*}
\frac{d}{dt}\int_\Omega |u-\overline{u_0}|^pdx+\alpha_2\int_\Omega|u-\overline{u_0}|^p\leq C\int_\Omega |\nabla v|^2+\alpha_2\int_\Omega|u-\overline{u_0}|^p.
\end{equation*}
Observe that
\begin{equation*}
\int_\Omega|u-\overline{u_0}|^p\leq \|u-\overline{u_0}\|^{p-2}_{L^\infty(\Omega)}\int_\Omega |u-\overline{u_0}|^2dx\leq C\int_\Omega |u-\overline{u_0}|^2dx.
\end{equation*}
We arrive at
\begin{equation}
\frac{d}{dt}\int_\Omega |u-\overline{u_0}|^pdx+\alpha_2\int_\Omega|u-\overline{u_0}|^p\leq C\left(\int_\Omega |\nabla v|^2+\int_\Omega|u-\overline{u_0}|^2\right)
\end{equation}
which yields that
\begin{equation}\label{udecayp}
\int_\Omega |u-\overline{u_0}|^pdx\leq Ce^{-\alpha_2t}.
\end{equation}
Note that from the second equation of \eqref{chemo1}
\begin{equation*}
(v-\overline{u_0})-\Delta(v-\overline{u_0})=u-\overline{u_0}.
\end{equation*}
Choosing some $p_0>n$ in \eqref{udecayp}, one may deduce by elliptic regularity and Sobolev embeddings that
\begin{equation*}
\|v-\overline{u_0}\|_{W^{1,\infty}(\Omega)}\leq C\|v-\overline{u_0}\|_{W^{2,p_0}(\Omega)}\leq C\|u-\overline{u_0}\|_{L^{p_0}(\Omega)}\leq Ce^{-\alpha t}.
\end{equation*}with $\alpha=\frac{\alpha_2}{p_0}.$ This completes the proof.
\end{proof}

Next, we claim that 
\begin{lemma}\label{ubound}
 There are positive constants $C=C(n,\Omega,\gamma,u_0)$ and $\theta\in(0,1)$ such that
\begin{equation}
\|u\|_{C^{2+\theta,1+\frac{\theta}{2}}(\overline{\Omega}\times[t,t+1])}\leq C\qquad\forall\;t\geq1.
\end{equation}
\end{lemma}
\begin{proof}
Since $\|u\|_{L^\infty(\Omega)}$ and $\|v\|_{W^{1,\infty}(\Omega)}$ are now uniform-in-time bounded, one infers from the key identity \eqref{keyid} and the second equation of \eqref{chemo1} that $v\in W^{2,1}_p(\overline{\Omega}\times[t,t+1])$ with any $p>\frac{n+2}{2} $ for any $t>0$. Then by the Sobolev embedding theorem, there exist $\theta_1\in(0,2-\frac{n+2}{p}]$ and time-independent constant $C>0$ such that \begin{equation}
\|v\|_{C^{\theta_1,\frac{\theta_1}{2}}(\overline{\Omega}\times[t,t+1])}\leq C\|v\|_{W^{2,1}_p(\overline{\Omega}\times[t,t+1])}\leq C\qquad\forall\;t>0.
\end{equation}
On the other hand, in the same manner as \cite[Lemma 5.1]{Anh19}, there exists time-independent constant $C>0$ such that \begin{equation*}
\|u\|_{C^{\theta_2,\frac{\theta_2}{2}}(\overline{\Omega}\times[t,t+1])}\leq C\qquad\forall\;t\geq1
\end{equation*}with some $\theta_2\in(0,1)$. 

Then in view  of the following variant form of the key identity:
\begin{equation*}
v_t-\gamma(v)\Delta v+v\gamma(v)=(I-\Delta)^{-1}[u\gamma(v)],
\end{equation*} we can further deduce by standard Schauder's theory for parabolic equations that with some $\theta_3\in(0,1)$
\begin{equation*}
\|v\|_{C^{2+\theta_3,1+\frac{\theta_3}{2}}(\overline{\Omega}\times[t,t+1])}\leq C\qquad\forall\;t\geq1.
\end{equation*}In turn, we may finally deduce from the equation for $u$ by Schauder's theory that 
\begin{equation*}
\|u\|_{C^{2+\theta,1+\frac{\theta}{2}}(\overline{\Omega}\times[t,t+1])}\leq C\qquad\forall\;t\geq1.
\end{equation*}
\end{proof}
With above preparations, we are now ready to prove the exponentially decay of $\|u-\overline{u_0}\|_{L^\infty}$. Denoting $w=u-\overline{u_0}$ and $\gamma_0=\gamma(\overline{u_0})$, by the semigroup theory, we infer from \eqref{ueqb} that for any $t>\tau_0\geq1,$
\begin{equation}
	w(t)=e^{\gamma_0\Delta t}w(\tau_0)+\int_{\tau_0}^t e^{\gamma_0\Delta (t-s)}\Delta((\gamma(v(s)-\gamma_0)u(s))ds.
\end{equation}
As a result, we deduce by Lemma \ref{lmpq} that
\begin{equation*}
	\begin{split}
	\|w(t)\|_{L^\infty(\Omega)}\leq& \|e^{\gamma_0\Delta t}w(\tau_0)\|_{L^\infty(\Omega)}+\int_{\tau_0}^t\|e^{\gamma_0\Delta (t-s)}\Delta((\gamma(v(s)-\gamma_0)u(s))\|_{L^\infty(\Omega)}ds\\
	\leq&\|e^{\gamma_0\Delta t}w(\tau_0)\|_{L^\infty(\Omega)}+C\int_{\tau_0}^te^{-\gamma_0\mu_1 (t-s)}(1+(t-s)^{-\frac12})\|\nabla((\gamma(v(s)-\gamma_0)u(s))\|_{L^\infty(\Omega)}ds
	\end{split}
\end{equation*}
Since $\|\nabla u\|_{L^\infty(\Omega)}\leq C$ for all $t\geq1$ due to Lemma \ref{ubound}, we obtain that for $t\geq1$
\begin{equation*}
\begin{split}
	\|\nabla((\gamma(v(t)-\gamma_0)u(t))\|_{L^\infty(\Omega)}\leq& \|\gamma'(v(t))u(t)\nabla v(t)\|_{L^\infty(\Omega)}+\|(\gamma(v(t))-\gamma(\overline{u_0}))\nabla u(t)\|_{L^\infty(\Omega)}\\
	\leq& C\|\nabla v\|_{L^\infty(\Omega)}+C\|\gamma(v(t))-\gamma(\overline{u_0})\|_{L^\infty(\Omega)}\\
	\leq&C(\|\nabla v\|_{L^\infty(\Omega)}+\|v(t)-\overline{u_0}\|_{L^\infty(\Omega)}
\end{split}
\end{equation*}
where we use the fact that
\begin{equation}
	|\gamma(v)-\gamma(\overline{u_0})|=|(v(t)-\overline{u_0})\int_0^1\gamma'(sv+(1-s)\overline{u_0})ds|\leq C|v(t)-\overline{u_0}|
\end{equation}since $sv(t,x)+(1-s)\overline{u_0}$ is uniformly bounded from above and below on $[0,+\infty)\times\overline{\Omega}$ for all $s\in[0,1]$.

As a result, recalling Lemma \ref{lmv1decay} and Lemma \ref{lmpq}, we may infer that
\begin{equation}\label{udecayinf}
\begin{split}
\|w(t)\|_{L^\infty(\Omega)}\leq& \|e^{\gamma_0\Delta t}w(\tau_0)\|_{L^\infty(\Omega)}+\int_{\tau_0}^t\|e^{\gamma_0\Delta (t-s)}\Delta((\gamma(v(s)-\gamma_0)u(s))\|_{L^\infty(\Omega)}ds\\
\leq&Ce^{-\gamma_0\mu_1 t}\|w(\tau_0)\|_{L^\infty(\Omega)}+C\int_{0}^te^{-\gamma_0\mu_1 (t-s)}(1+(t-s)^{-\frac12})e^{-\alpha s}ds\\
\leq& Ce^{-\alpha't}
\end{split}
\end{equation}with any $\alpha'<\min\{\gamma_0\mu_1,\alpha\}$. Here, we  use the fact that for any $\beta\geq\kappa>0$
\begin{equation*}
\begin{split}
\int_{0}^te^{-\beta (t-s)}(1+(t-s)^{-\frac12})e^{-\kappa s}ds
=&e^{-\beta t}\int_0^{t}e^{(\beta-\kappa) s}(1+(t-s)^{-\frac12})ds\\
\leq&e^{-\kappa t}\bigg(t+2t^{\frac12}\bigg)\leq Ce^{-\kappa't}
\end{split}
\end{equation*}with any $\kappa'<\kappa$ and on the other hand, for $0<\beta<\kappa$
\begin{equation*}
\begin{split}
\int_{0}^te^{-\beta (t-s)}(1+(t-s)^{-\frac12})e^{-\kappa s}ds
=&e^{-\beta t}\int_0^{t}e^{(\beta-\kappa) s}(1+(t-s)^{-\frac12})ds\\
\leq&e^{-\beta t}\bigg(t+2t^{\frac12}\bigg)\leq Ce^{-\beta't}
\end{split}
\end{equation*}with any $\beta'<\beta.$

\noindent\textbf{Proof of Theorem \ref{TH1}. Convergence:} By Lemma \ref{lmv1decay} and \eqref{udecayinf}, we conclude that
\begin{equation*}
\|u(\cdot,t)-\overline{u_0}\|_{L^\infty(\Omega)}+\|v(\cdot,t)-\overline{u_0}\|_{W^{1,\infty}(\Omega)}\leq Ce^{-\alpha' t},\;\;\forall \;t\geq1
\end{equation*}
with some $\alpha'>0$ and $C>0$ depending on $u_0,\gamma,n$ and $\Omega$.

\bigskip
\noindent\textbf{Acknowledgments} \\
This work was supported by Hubei Provincial Natural Science Foundation under the grant No. 2020CFB602.

\end{document}